\documentclass[12pt,a4paper,reqno,twoside]{amsart}
\usepackage{amsmath,amssymb,amsfonts,amscd,amsthm}

\usepackage{xcolor}

\newtheorem{theorem}{Theorem}[section]
\newtheorem{corollary}{Corollary}[section]
\newtheorem{lemma}{Lemma}[section]

\theoremstyle{definition}
\newtheorem{definition}{Definition}[section]

\theoremstyle{remark}
\newtheorem{remark}{Remark}[section]

\newtheorem{example}{Example}[section]

\numberwithin{equation}{section}
\begin{document}

\title[Bi-univalent Functions of Complex Order]{Estimate for Initial MacLaurin Coefficients\\
of Certain Subclasses of Bi-univalent Functions of\\
Complex Order Associated with the Hohlov Operator}

\author[T. Bulboac\u{a}]{Teodor Bulboac\u{a}}

\address{Corresponding Author,Faculty of Mathematics and Computer Science, Babe\c{s}-Bolyai University, 400084 Cluj-Na\-po\-ca, Romania}

\email{bulboaca@math.ubbcluj.ro}
\author[G. Murugusundaramoorthy]{G. Murugusundaramoorthy}

\address{Department of Mathematics, School of Advanced Sciences, VIT University, Vellore 632014, Tamilnadu, India}

\email{gmsmoorthy@yahoo.com}
\begin{abstract}
In this paper we introduce and investigate two new subclasses of the function class $\Sigma$ of bi-univalent functions of complex order defined in the open unit disk, which are associated with the Hohlov operator, and satisfying subordinate conditions. Furthermore, we find estimates on the Taylor-MacLaurin coefficients $|a_2|$ and $|a_3|$ for functions in these new subclasses. Several known or new consequences of these results are also pointed out.
\end{abstract}



\maketitle

\section{Introduction}

Let $\mathcal{A}$ denote the class of functions of the form
\begin{equation}\label{Int-e1}
f(z)=z+\sum\limits_{n=2}^{\infty}a_nz^n,
\end{equation}
which are analytic in the open unit disk $\mathbb{U}:=\{z\in\mathbb{C}:|z|<1\}$.

By $\mathcal{S}$ we will denote the subclass of all functions in $\mathcal{A}$ which are univalent in $\mathbb{U}$. Some of the important and well-investigated subclasses of the class $\mathcal{S}$ include, for example, the class $\mathcal{S}^*(\alpha)$ of {\em starlike functions of order} $\alpha$ in $\mathbb{U}$, and the class $\mathcal{K}(\alpha)$ of {\em convex functions of order} $\alpha$ in $\mathbb{U}$, with $0\leq\alpha<1$.

It is well known that every function $f\in\mathcal{S}$ has an inverse $f^{-1}$, defined by
\[
f^{-1}(f(z))=z\quad(z\in\mathbb{U})
\]
and
\[
f(f^{-1}(w))=w\quad\left(|w|<r_0(f),\;r_0(f)\geq\frac{1}{4}\right),
\]
where
\begin{equation}\label{g-e}
g(w)=f^{-1}(w)=w-a_2w^2+\left(2a_2^2-a_3\right)w^3-\left(5a_2^3-5a_2a_3+a_4\right)w^4+\dots.
\end{equation}

A function $f\in\mathcal{A}$ is said to be {\em bi-univalent} in $\mathbb{U}$ if $f(z)$ and $f^{-1}(w)$ are univalent in $\mathbb{U}$, and let $\Sigma$ denote the class of {\em bi-univalent functions} in $\mathbb{U}$.

The {\em convolution} or {\em Hadamard product} of two functions $f,h\in\mathcal{A}$ is denoted by $f\ast h$, and is defined by
\[
(f\ast h)(z):=z+\sum\limits_{n=2}^{\infty }a_{n}b_{n}z^{n},
\]
where $f$ is given by \eqref{Int-e1} and $h(z)=z+\sum\limits_{n=2}^{\infty }b_{n}z^{n}$. Next, in our present investigation, we need to recall the {\em convolution operator} $\mathcal{I}_{a,b,c}$ due to Hohlov \cite{HO1,HO2}, which is a special case of the {\em Dziok-Srivastava operator} \cite{dziok2,dziok1}.

For the complex parameters $a$, $b$ and $c$ $(c\neq0,-1,-2,-3,\dots)$, the {\em Gaussian hypergeometric function} $\,_{2}F_{1}(a,b,c;z)$ is defined as
\begin{equation}\label{Gauss-equ}
\,_{2}F_{1}(a,b,c;z):=\sum_{n=0}^{\infty}\frac{(a)_{n}(b)_{n}}{(c)_{n}}\;\frac{z^{n}}{n!}
=1+\sum_{n=2}^{\infty}\frac{(a)_{n-1}(b)_{n-1}}{(c)_{n-1}}\;\frac{z^{n-1}}{(n-1)!}\quad(z\in\mathbb{U}),
\end{equation}
where $(\alpha)_{n}$ is the {\em Pochhammer symbol} (or the {\em shifted factorial}) given by
\[
(\alpha)_{n}:=\frac{\Gamma(\alpha+n)}{\Gamma(\alpha)}=\left\{
\begin{array}{lll}
1,&\text{if}&n=0,\\
\alpha(\alpha+1)(\alpha+2)\cdots(\alpha+n-1),&\text{if}&n=1,2,3,\dots.
\end{array}
\right.
\]

For the real positive values $a$, $b$ and $c$, using the Gaussian hypergeometric function \eqref{Gauss-equ}, Hohlov \cite{HO1,HO2} introduced the familiar convolution operator $\mathcal{I}_{a,b,c}:\mathcal{A}\rightarrow\mathcal{A}$ by
\begin{equation}\label{Hohlov}
\mathcal{I}_{a,b,c}f(z)=\left[z\,_{2}F_{1}(a,b,c;z)\right]\ast f(z)=z+\sum_{n=2}^{\infty}\varphi_{n}a_{n}z^{n}
\quad(z\in\mathbb{U}),
\end{equation}
where
\begin{equation}\label{varphi-n}
\varphi_{n}=\frac{(a)_{n-1}(b)_{n-1}}{(c)_{n-1}(n-1)!},
\end{equation}
and the function $f$ is of the form \eqref{Int-e1}.

Hohlov \cite{HO1,HO2} discussed some interesting geometrical properties exhibited by the operator $\mathcal{I}_{a,b,c}$, and the three-parameter family of operators $\mathcal{I}_{a,b,c}$ contains, as its special cases, most of the known linear integral or differential operators. In particular, if
$b=1$ in \eqref{Hohlov}, then $\mathcal{I}_{a,b,c}$ reduces to the {\em Carlson-Shaffer operator}. Similarly, it is easily seen that the {\em Hohlov operator} $\mathcal{I}_{a,b,c}$ is also a generalization of the {\em Ruscheweyh derivative operator} as well as the {\em Bernardi-Libera-Livingston operator}. It is of interest to note that for $a=c$ and $b=1$, then $\mathcal{I}_{a,1,a}f=f$, for all $f\in\mathcal{A}$.

Recently there has been triggering interest to study bi-univalent function class $\Sigma$ and obtained non-sharp coefficient estimates on the first two coefficients $|a_2|$ and $|a_3|$ of \eqref{Int-e1}. But the coefficient problem for each of the following {\em Taylor-MacLaurin coefficients}
$$|a_n|\quad\left(n\geq3\right)$$
is still an open problem (see \cite{Branna1970,Bran-1979,Bran1985,Lewin,Netany,Taha1981}). Many researchers (see \cite{BAF-MKA,haya,Li-Wang,HMS-AKM-PG}) have recently introduced and investigated several interesting subclasses of the bi-univalent function class $\Sigma$ and they have found non-sharp estimates on the
first two Taylor-MacLaurin coefficients $|a_2|$ and $|a_3|$.

\section{Definitions and Preliminaries}

In \cite{pad} the authors defined the classes of functions $\mathcal{P}_{m}(\beta)$ as follows:

\begin{definition}\cite{pad}
Let $\mathcal{P}_{m}(\beta)$, with $m\geq2$ and $0\leq\beta<1$, denote the class of univalent analytic functions $P$, normalized with $P(0)=1$, and satisfying
\[
\int_0^{2\pi}\left|\frac{\operatorname{Re}P(z)-\beta}{1-\beta}\right|\operatorname{d}\theta\leq m\pi,
\]
where $z=re^{i\theta}\in\mathbb{U}$.
\end{definition}

For $\beta=0$, we denote $\mathcal{P}_m:=\mathcal{P}_m(0)$, hence the class $\mathcal{P}_m$ represents the class of functions $p$ analytic in $\mathbb{U}$, normalized with $p(0)=1$, and having the representation
\[
p(z)=\int\limits_0^{2\pi}\frac{1- ze^{it}}{1+ze^{it}}\operatorname{d}\mu(t),
\]
where $\mu$ is a real-valued function with bounded variation, which satisfies
\[
\int_0^{2\pi}d\mu(t)=2\pi\quad\text{and}\quad\int_0^{2\pi}\left|d\mu(t)\right|\leq m,\;m\geq2.
\]
Remark that $\mathcal{P}:=\mathcal{P}_2$ is the well-known class of {\em Carath\'eodory functions}, i.e. the normalized functions with positive real part in the open unit disk $\mathbb{U}$.

Motivated by the earlier work of Deniz \cite{DEN}, Peng et al. \cite{Peng} (see also \cite{GMS-TP,HMS-GMS-NM}) and Goswami et al. \cite{gosami}, in the present paper we introduce new subclasses of the function class $\Sigma$ of complex order $\gamma\in\mathbb{C}^*:=\mathbb{C}\setminus\{0\}$, involving Hohlov operator $\mathcal{I}_{a,b,c}$, and we find estimates on the coefficients $|a_2|$ and $|a_3|$ for the functions that belong to these new subclasses of functions of the class $\Sigma$. Several related classes are also considered, and connection to earlier known results are made.

\begin{definition}\label{Defi-1}
For $ 0\leq\lambda\leq1$ and $0\leq\beta<1$, a function $f\in\Sigma$ is said to be in the
class $\mathcal{S}^{a,b,c}_{\Sigma}(\gamma,\lambda,\beta)$ if the following two conditions are satisfied:
\begin{equation}\label{Defi-1-e1}
1+\frac{1}{\gamma}\left[\frac{z\left(\mathcal{I}_{a,b,c}f(z)\right)'}{(1-\lambda)z+
\lambda\mathcal{I}_{a,b,c}f(z)}-1\right]\in\mathcal{P}_m(\beta)
\end{equation}
and
\begin{equation}\label{Defi-1-e2}
1+\frac{1}{\gamma}\left[\frac{w\left(\mathcal{I}_{a,b,c}g(w)\right)'}{(1-\lambda)w+
\lambda\mathcal{I}_{a,b,c}g(w)}-1\right]\in\mathcal{P}_m(\beta),
\end{equation}
where $\gamma\in\mathbb{C}^*$, the function $g$ is given by \eqref{g-e}, and $z,w\in\mathbb{U}$.
\end{definition}

\begin{definition}\label{defin2}
For $0\leq\lambda\leq1$ and $0\leq\beta<1$, a function $f\in\Sigma$ is said to be in the class $\mathcal{K}^{a,b,c}_{\Sigma}(\gamma,\lambda,\beta)$ if it satisfies the following two conditions:
\begin{equation}\label{def2}
1+\frac{1}{\gamma}\left[\frac{z\left(\mathcal{I}_{a,b,c}f(z)\right)'+z^2\left(\mathcal{I}_{a,b,c}f(z)\right)''}
{(1-\lambda)z+\lambda z\left(\mathcal{I}_{a,b,c}f(z)\right)'}-1\right]\in\mathcal{P}_m(\beta)
\end{equation}
and
\begin{equation}\label{def21}
1+\frac{1}{\gamma}\left[\frac{w\left(\mathcal{I}_{a,b,c}g(w)\right)'+w^2\left(\mathcal{I}_{a,b,c}g(w)\right)''}
{(1-\lambda)w+\lambda w\left(\mathcal{I}_{a,b,c}g(w)\right)'}-1\right]\in\mathcal{P}_m(\beta),
\end{equation}
where $\gamma\in\mathbb{C}^*$, the function $g$ is given by \eqref{g-e}, and $z,w\in\mathbb{U}$.
\end{definition}

On specializing the parameters $\lambda$ one can state the various new subclasses of $\Sigma$ as illustrated in the following examples. Thus, taking $\lambda=1$ in the above two definitions, we obtain:

\begin{example}\label{ex1}
Suppose that $0\leq\beta<1$ and $\gamma\in\mathbb{C}^*$.

(i) A function $f\in\Sigma$ is said to be in the class $\mathcal{S}^{a,b,c}_{\Sigma}(\gamma,\beta)$ if the following conditions are satisfied:
\[
1+\frac{1}{\gamma}\left[\frac{z\left(\mathcal{I}_{a,b,c}f(z)\right)'}{\mathcal{I}_{a,b,c}f(z)}-1\right]
\in\mathcal{P}_m(\beta),\quad
1+\frac{1}{\gamma}\left[\frac{w\left(\mathcal{I}_{a,b,c}g(w)\right)'}{\mathcal{I}_{a,b,c}g(w)}-1\right]
\in\mathcal{P}_m(\beta),
\]
where $g=f^{-1}$ and $z,w\in\mathbb{U}$.

(ii) A function $f\in\Sigma$ is said to be in the class $\mathcal{K}^{a,b,c}_{\Sigma}(\gamma,\beta)$ if it satisfies the following conditions:
\[
1+\frac{1}{\gamma}\frac{z\left(\mathcal{I}_{a,b,c}f(z)\right)''}
{\left(\mathcal{I}_{a,b,c}f(z)\right)'}\in\mathcal{P}_m(\beta),\quad
1+\frac{1}{\gamma}\frac{w\left(\mathcal{I}_{a,b,c}g(w)\right)''}{\left(\mathcal{I}_{a,b,c}g(w)\right)'}
\in\mathcal{P}_m(\beta),
\]
where $g=f^{-1}$ and $z,w\in\mathbb{U}$.
\end{example}

Taking $\lambda=0$ in the previous two definitions, we obtain the next special cases:

\begin{example}\label{ex3}
Suppose that $0\leq\beta<1$ and $\gamma\in\mathbb{C}^*$.

(i) A function $f\in\Sigma$ is said to be in the class $\mathcal{H}^{a,b,c}_{\Sigma}(\gamma,\beta)$ if the following conditions are satisfied:
\[
1+\frac{1}{\gamma}\left[\left(\mathcal{I}_{a,b,c}f(z)\right)'-1\right]\in\mathcal{P}_m(\beta),\quad
1+\frac{1}{\gamma}\left[\left(\mathcal{I}_{a,b,c}g(w)\right)'-1\right]\in\mathcal{P}_m(\beta),
\]
where $g=f^{-1}$ and $z,w\in\mathbb{U}$.

(ii) A function $f\in\Sigma$ is said to be in the class $\mathcal{Q}^{a,b,c}_{\Sigma}(\gamma,\beta)$ if it satisfies the following conditions:
\begin{eqnarray*}
&&1+\frac{1}{\gamma}\left[\left(\mathcal{I}_{a,b,c}f(z)\right)'+z\left(\mathcal{I}_{a,b,c}f(z)\right)''-1\right]
\in\mathcal{P}_m(\beta),\\
&&1+\frac{1}{\gamma}\left[\left(\mathcal{I}_{a,b,c}g(w)\right)'+w\left(\mathcal{I}_{a,b,c}g(w)\right)''-1\right]
\in\mathcal{P}_m(\beta),
\end{eqnarray*}
where $g=f^{-1}$ and $z,w\in\mathbb{U}$.
\end{example}

In particular, for $a=c$ and $b=1$, we note that $\mathcal{I}_{a,1,a}f=f$ for all $f\in\mathcal{A}$, and thus, for $\lambda=1$ and $\lambda=0$ the classes $\mathcal{S}^{a,b,c}_{\Sigma}(\gamma,\lambda,\beta)$ and $\mathcal{K}^{a,b,c}_{\Sigma}(\gamma,\lambda,\beta)$ reduces to the following subclasses of $\Sigma$, respectively:

\begin{example}\label{ex5}
(i) For $0\leq\beta<1$ and $\gamma\in\mathbb{C}^*$, a function $f\in\Sigma$ is said to be in the
class $\mathcal{S}^*_{\Sigma}(\gamma,\beta)$ if the following conditions are satisfied:
\[
1+\frac{1}{\gamma}\left(\frac{z f'(z)}{f(z)}-1\right)\in\mathcal{P}_m(\beta)\quad\text{and}\quad
1+\frac{1}{\gamma}\left(\frac{w g'(w)}{g(w)}-1\right)\in\mathcal{P}_m(\beta),
\]
where $g=f^{-1}$ and $z,w\in\mathbb{U}$.

(ii) For $0\leq\beta<1$ and $\gamma\in\mathbb{C}^*$, a function $f\in\Sigma$ is said to be in the
class $\mathcal{K}_{\Sigma}(\gamma,\beta)$ if the following conditions are satisfied:
\[
1+\frac{1}{\gamma}\frac{zf''(z)}{f'(z)}\in\mathcal{P}_m(\beta)\quad\text{and}\quad
1+\frac{1}{\gamma}\frac{w g''(w)}{g'(w)}\in\mathcal{P}_m(\beta),
\]
where $g=f^{-1}$ and $z,w\in\mathbb{U}$.
\end{example}

\begin{example}\label{exam6}

(i) For $0\leq\beta<1$ and $\gamma\in\mathbb{C}^*$, a function $f\in\Sigma$ is said to be in the class $\mathcal{H}_{\Sigma}(\gamma,\beta)$ if the following conditions are satisfied:
\[
1+\frac{1}{\gamma}\left(f'(z)-1\right)\in\mathcal{P}_m(\beta)\quad\text{and}\quad
1+\frac{1}{\gamma}\left(g'(w)-1\right)\in\mathcal{P}_m(\beta),
\]
where $g=f^{-1}$ and $z,w\in\mathbb{U}$.

(ii) For $0\leq\beta<1$ and $\gamma\in\mathbb{C}^*$, a function $f\in\Sigma$ is said to be in the class $\mathcal{Q}_{\Sigma}(\gamma,\beta)$ if the following conditions are satisfied:
\[
1+\frac{1}{\gamma}\left(f'(z)+zf''(z)-1\right)\in\mathcal{P}_m(\beta)\quad\text{and}\quad
1+\frac{1}{\gamma}\left(g'(w)+wg''(w)-1\right)\in\mathcal{P}_m(\beta),
\]
where $g=f^{-1}$ and $z,w\in\mathbb{U}$.
\end{example}

In order to derive our main results, we shall need the following lemma:

\begin{lemma}\cite[Lemma 2.1]{gosami}\label{lem2.1}
Let the function $\Phi(z)=1+\sum\limits_{n=1}^\infty{h_n}{z^n}$, $z\in\mathbb{U}$, such that $\Phi\in\mathcal{P}_m(\beta)$. Then,
\[
\left|h_n\right|\leq m(1-\beta),\;n\geq1.
\]
\end{lemma}

By employing the techniques used earlier by Deniz \cite{DEN}, in the following section we find estimates of the coefficients $|a_2|$ and $|a_3|$ for functions of the above-defined subclasses $\mathcal{S}^{a,b,c}_{\Sigma}(\gamma,\lambda,\beta)$ and $\mathcal{K}^{a,b,c}_{\Sigma}(\gamma,\lambda,\beta)$ of the function class $\Sigma$.

\section{Coefficient Bounds for the Function Class $\mathcal{S}^{a,b,c}_{\Sigma}(\gamma,\lambda,\beta)$}

We begin by finding the estimates on the coefficients $|a_2|$ and $|a_3|$ for functions belonging to the class
$\mathcal{S}^{a,b,c}_{\Sigma}(\gamma,\lambda,\beta)$.

Supposing that the functions $p,q\in\mathcal{P}_m(\beta)$, with
\begin{eqnarray}
&&p(z)=1+\sum_{k=1}^{\infty}p_kz^k\quad(z\in\mathbb{U}),\label{phi-uz}\\
&&q(z)=1+\sum_{k=1}^{\infty}q_kz^k\quad(z\in\mathbb{U}),\label{phi-vw}
\end{eqnarray}
from Lemma \ref{lem2.1} it follows that
\begin{eqnarray}
&&|p_k|\leq m(1-\beta),\label{mod-p-q}\\
&&|q_k|\leq m(1-\beta),\quad\text{for all}\quad k\geq1.\label{c7e2.14}
\end{eqnarray}

\begin{theorem}\label{Bi-th1}
If the function $f$ given by \eqref{Int-e1} belongs to the class $\mathcal{S}^{a,b,c}_{\Sigma}(\gamma,\lambda,\beta)$, then
\begin{equation}\label{bi-th1-b-a2}
|a_2|\leq\min\left\{\sqrt{\frac{m|\gamma|(1-\beta)}
{\left|(\lambda^2-2\lambda)\varphi_2^2+(3-\lambda)\varphi_3\right|}};
\frac{m|\gamma|(1-\beta)}{(2-\lambda)\varphi_2}\right\}
\end{equation}
and
\begin{eqnarray}\label{bi-th1-b-a3}
|a_3|\leq\min\left\{\dfrac{m|\gamma|(1-\beta)}{(3-\lambda)\varphi_3}+\dfrac{m|\gamma|(1-\beta)}
{\left|(\lambda^2-2\lambda)\varphi_2^2+(3-\lambda)\varphi_3\right|};\right.\nonumber\\
\dfrac{m|\gamma|(1-\beta)}{(3-\lambda)\varphi_3}
\left(1+\dfrac{m|\gamma|(2\lambda-\lambda^2)(1-\beta)}{(2-\lambda)^2\varphi_2^2}\right);\nonumber\\
\left.\dfrac{m|\gamma|(1-\beta)}{(3-\lambda)\varphi_3}\left(1+m|\gamma|(1-\beta)
\dfrac{\left|(\lambda^2-2\lambda)\varphi_2^2+2(3-\lambda)\varphi_3\right|}
{(2-\lambda)^2\varphi_2^2}\right)\right\},
\end{eqnarray}
where $\varphi_2$ and $\varphi_3$ are given by \eqref{varphi-n}.
\end{theorem}

\begin{proof}
Since $f\in\mathcal{S}^{a,b,c}_{\Sigma}(\gamma,\lambda,\beta)$, from the definition relations \eqref{Defi-1-e1} and \eqref{Defi-1-e2} it follows that
\begin{eqnarray}\label{bi-th1-pr-e1}
&1+\dfrac{1}{\gamma}\left[\dfrac{z\left(\mathcal{I}_{a,b,c}f(z)\right)'}
{(1-\lambda)z+\lambda\mathcal{I}_{a,b,c}f(z)}-1\right]=\nonumber\\
&1+\dfrac{2-\lambda}{\gamma}\varphi_2 a_2 z+\left[\dfrac{\lambda^2-2\lambda}{\gamma}\varphi_2^2 a_2^2+\dfrac{3-\lambda}{\gamma}\varphi_3 a_3\right]z^2+\dots=:p(z)
\end{eqnarray}
and
\begin{eqnarray}\label{bi-th1-pr-e2}
&1+\dfrac{1}{\gamma}\left[\dfrac{w\left(\mathcal{I}_{a,b,c}g(w)\right)'}
{(1-\lambda)w+\lambda\mathcal{I}_{a,b,c}g(w)}-1\right]=\nonumber\\
&1-\dfrac{2-\lambda}{\gamma}\varphi_2 a_2 w+\left[\dfrac{\lambda^2-2\lambda}{\gamma}\varphi_2^2 a_2^2+\dfrac{3-\lambda}{\gamma}\varphi_3\left(2a_2^2-a_3\right)\right]w^2+\dots=:q(w),
\end{eqnarray}
where $p,q\in\mathcal{P}_m(\beta)$, and are of the form \eqref{phi-uz} and \eqref{phi-vw}, respectively.

Now, equating the coefficients in \eqref{bi-th1-pr-e1} and \eqref{bi-th1-pr-e2}, we get
\begin{eqnarray}
&&\frac{2-\lambda}{\gamma}\varphi_2 a_2=p_{1},\label{th1-ceof-p1}\\
&&\frac{\lambda^2-2\lambda}{\gamma}\varphi_2^2 a_2^2+
\frac{3-\lambda}{\gamma}\varphi_3 a_3=p_{2},\label{th1-ceof-p2}\\
&&-\frac{2-\lambda}{\gamma}\varphi_2 a_2=q_{1},\label{th1-ceof-q1}
\end{eqnarray}
and
\begin{equation}\label{th1-ceof-q2}
\frac{\lambda^2-2\lambda}{\gamma}\varphi_2^2 a_2^2+
\frac{3-\lambda}{\gamma}\varphi_3\left(2a_2^2-a_3\right)=q_{2}.
\end{equation}
From \eqref{th1-ceof-p1} and \eqref{th1-ceof-q1}, we find that
\begin{equation}\label{th1-pr-p1=q1-a2}
a_2=\frac{\gamma p_1}{(2-\lambda)\varphi_2}=\frac{-\gamma q_1}{(2-\lambda)\varphi_2},
\end{equation}
which implies
\begin{equation}\label{th1-pr-p1=q1}
|a_2|\leq \frac{|\gamma| m(1-\beta)}{(2-\lambda)\varphi_2}.
\end{equation}
Adding \eqref{th1-ceof-p2} and \eqref{th1-ceof-q2}, by using \eqref{th1-pr-p1=q1-a2} we obtain
\[
\left[2\left(\lambda^2-2\lambda\right)\varphi_2^2+2(3-\lambda)\varphi_3\right]a_2^2=\gamma(p_2+q_2).
\]
Now, by using \eqref{mod-p-q} and \eqref{c7e2.14}, we get
\[
|a_2|^2=\frac{m|\gamma|(1-\beta)}{\left|\left(\lambda^2-2\lambda\right)\varphi_2^2+(3-\lambda)\varphi_3\right|},
\]
hence
\[
|a_2|\leq\sqrt{\frac{m|\gamma|(1-\beta)}{\left|(\lambda^2-2\lambda)\varphi_2^2+(3-\lambda)\varphi_3\right|}},
\]
which gives the bound on $|a_2|$ as asserted in \eqref{bi-th1-b-a2}.

Next, in order to find the upper-bound for $|a_3|$, by subtracting \eqref{th1-ceof-q2} from \eqref{th1-ceof-p2}, we get
\begin{equation}\label{th1-a3-cal-e1}
2(3-\lambda)\varphi_3 a_3=\gamma(p_2-q_2)+2(3-\lambda)\varphi_3a_2^2.
\end{equation}
It follows from \eqref{mod-p-q}, \eqref{th1-pr-p1=q1} and \eqref{th1-a3-cal-e1}, that
\[
|a_3|\leq\frac{m|\gamma|(1-\beta)}{(3-\lambda)|\varphi_3|}+\frac{m|\gamma|(1-\beta)}
{\left|(\lambda^2-2\lambda)\varphi_2^2+(3-\lambda)\varphi_3\right|}.
\]
From \eqref{th1-ceof-p1} and \eqref{th1-ceof-p2} we have
\[
a_3=\frac{1}{(3-\lambda)\varphi_3}
\left(\gamma p_2-\frac{\gamma^2(\lambda^2-2\lambda)p^2_1}{(2-\lambda)^2\varphi_2^2}\right),\]
hence
\[
|a_3|\leq\frac{m|\gamma|(1-\beta)}{(3-\lambda)\varphi_3}
\left(1+\frac{m|\gamma(\lambda^2-2\lambda)|(1-\beta)}{(2-\lambda)^2\varphi_2^2}\right).
\]
Further, from \eqref{th1-ceof-p1} and \eqref{th1-ceof-q2} we deduce that
\[
|a_3|\leq\frac{m|\gamma|(1-\beta)}{(3-\lambda)\varphi_3}\left(1+m|\gamma|(1-\beta) \frac{|(\lambda^2-2\lambda)\varphi_2^2+2(3-\lambda)\varphi_3|}{(2-\lambda)^2\varphi_2^2}\right),
\]
and thus we obtain the conclusion \eqref{bi-th1-b-a3} of our theorem.
\end{proof}

For the special cases $\lambda=1$ and $\lambda=0$, the Theorem \ref{Bi-th1} reduces to the following corollaries, respectively:

\begin{corollary}\label{sss1}
If the function $f$ given by \eqref{Int-e1} belongs to the class $\mathcal{S}^{a,b,c}_{\Sigma}(\gamma,\beta)$, then
\[
\left|a_2\right|\leq\min\left\{\sqrt{\frac{m|\gamma|(1-\beta)}{\left|2\varphi_3-\varphi_2^2\right|}};
\frac{m|\gamma|(1-\beta)}{\varphi_2}\right\}
\]
and
\begin{eqnarray*}
\left|a_3\right|\leq\min\left\{\dfrac{m|\gamma|(1-\beta)}{\left|2\varphi_3-\varphi_2^2\right|}+
\dfrac{m|\gamma|(1-\beta)}{2\varphi_3};
\dfrac{m|\gamma|(1-\beta)}{2\varphi_3}\left(1+\dfrac{m|\gamma|(1-\beta)}{\varphi_2^2}\right);\right.\\
\left.\dfrac{m|\gamma|(1-\beta)}{2\varphi_3}
\left(1+\dfrac{m|\gamma|(1-\beta)\left|4\varphi_3-\varphi_2^2\right|}{\varphi_2^2}\right)\right\},
\end{eqnarray*}
where $\varphi_2$ and $\varphi_3$ are given by \eqref{varphi-n}.
\end{corollary}

\begin{corollary}\label{Bi-cor1}
If the function $f$ given by \eqref{Int-e1} belongs to the class $\mathcal{G}^{a,b,c}_{\Sigma}(\gamma,\beta)$, then
\[
|a_2|\leq\min\left\{\sqrt{\frac{m|\gamma|(1-\beta)}{3\varphi_3}};\frac{m|\gamma|(1-\beta)}{2\varphi_2}\right\}
\]
and
\[
|a_3|\leq\dfrac{m|\gamma|(1-\beta)}{3\varphi_3},
\]
where $\varphi_2$ and $\varphi_3$ are given by \eqref{varphi-n}.
\end{corollary}

\section{Coefficient Bounds for the Function Class $\mathcal{K}^{a,b,c}_{\Sigma}(\gamma,\lambda,\beta)$}

\begin{theorem}\label{thgms1}
If the function $f$ given by \eqref{Int-e1} belongs to the class $\mathcal{K}^{a,b,c}_{\Sigma}(\gamma,\lambda,\beta)$, then
\begin{equation}\label{K-a2}
|a_2|\leq\min\left\{\sqrt{\frac{m|\gamma|(1-\beta)}
{\left|4(\lambda^2-2\lambda)\varphi_2^2+3(3-\lambda)\varphi_3\right|}};
\frac{m|\gamma|(1-\beta)}{2(2-\lambda)\varphi_2}\right\}
\end{equation}
and
\begin{eqnarray}\label{K-a3}
|a_3|\leq\min\left\{\frac{m|\gamma|(1-\beta)}{3(3-\lambda)\varphi_3}
\left(1+\frac{m|\gamma|(2\lambda-\lambda^2)(1-\beta)}{(2-\lambda)^2\varphi_2^2}\right);\right.\nonumber\\
\left.\frac{m|\gamma|(1-\beta)}{3(3-\lambda)\varphi_3}+\frac{m|\gamma|(1-\beta)}
{\left|4(\lambda^2-2\lambda)\varphi_2^2+3(3-\lambda)\varphi_3\right|};\nonumber\right.\\
\left.\frac{m|\gamma|(1-\beta)}{3(3-\lambda)\varphi_3}+ \frac{m^2|\gamma|^2(1-\beta)^2}{3(3-\lambda)\varphi_3}\left(1+ \frac{3(3-\lambda)\varphi_3}{2(2-\lambda)^2\varphi_2^2}\right)\right\},
\end{eqnarray}
where $\varphi_2$ and $\varphi_3$ are given by \eqref{varphi-n}.
\end{theorem}

\begin{proof}
For $f\in\mathcal{K}^{a,b,c}_{\Sigma}(\gamma,\lambda,\beta)$, from the definition relations \eqref{def2} and \eqref{def21} it follows that
\begin{eqnarray}\label{GMS1}
&1+\dfrac{1}{\gamma}\left[\dfrac{z\left(\mathcal{I}_{a,b,c}f(z)\right)'+
z^2\left(\mathcal{I}_{a,b,c}f(z)\right)''}{(1-\lambda)z+\lambda z\left(\mathcal{I}_{a,b,c}f(z)\right)'}-1\right]=\nonumber\\
&1+\dfrac{2(2-\lambda)}{\gamma}\varphi_2 a_2 z+\left[\dfrac{4(\lambda^2-2\lambda)}{\gamma}\varphi_2^2 a_2^2+\dfrac{3(3-\lambda)}{\gamma}\varphi_3 a_3\right]z^2+\dots=:p(z)
\end{eqnarray}
and
\begin{eqnarray}\label{GMS2}
&1+\dfrac{1}{\gamma}\left[\dfrac{w\left(\mathcal{I}_{a,b,c}g(w)\right)'+
w^2\left(\mathcal{I}_{a,b,c}g(w)\right)''}{(1-\lambda)w+\lambda z\left(\mathcal{I}_{a,b,c}g(w)\right)'}-1\right]\nonumber=\\
&1-\dfrac{2(2-\lambda)}{\gamma}\varphi_2 a_2 w+\left[\dfrac{4(\lambda^2-2\lambda)}{\gamma}\varphi_2^2 a_2^2+\dfrac{3(3-\lambda)}{\gamma}\varphi_3(2a_2^2-a_3)\right]w^2+\dots=:q(w),
\end{eqnarray}
where $p,q\in\mathcal{P}_m(\beta)$, and are of the form \eqref{phi-uz} and \eqref{phi-vw}, respectively.

Now, equating the coefficients in \eqref{GMS1} and \eqref{GMS2}, we get
\begin{eqnarray}
&&\frac{2}{\gamma}(2-\lambda)\varphi_2 a_{2}=p_{1},\label{e3.5}\\
&&\frac{1}{\gamma}\left[4(\lambda^2-2\lambda)\varphi^2_2a^2_2+3(3-\lambda)\varphi_3a_{3}\right]=
p_{2},\label{e3.6}\\
&&-\frac{2}{\gamma}(2-\lambda)\varphi_2 a_{2}=q_{1},\nonumber
\end{eqnarray}
and
\begin{equation}\label{e3.8}
\frac{1}{\gamma}\left[4(\lambda^2-2\lambda)\varphi^2_2 a^2_2+3(3-\lambda)(2a_2^2-a_3)\varphi_3\right]=q_{2}.
\end{equation}
From \eqref{e3.5} we get
\begin{equation}\label{e3.9}
a_2=\frac{p_1\gamma}{2(2-\lambda)\varphi_2},
\end{equation}
further, by adding \eqref{e3.6} and \eqref{e3.8}, and using \eqref{e3.9} we get
\begin{equation}\label{e3.10}
a_2^2=\frac{(p_2+q_2)\gamma}{8(\lambda^2-2\lambda)\varphi_2^2+6(3-\lambda)\varphi_3}.
\end{equation}
Now, from \eqref{e3.5} and \eqref{e3.10}, according to Lemma \ref{lem2.1} we easily deduce the inequality \eqref{K-a2}.

Next, in order to find the upper-bound for $|a_3|$, from \eqref{e3.6}, by using \eqref{e3.9} we have
\[
a_3=\frac{p_2\gamma}{3(3-\lambda)\varphi_3}-\frac{(\lambda^2-2\lambda)p_1^2\gamma^2}{3(2-\lambda)^2(3-\lambda)\varphi_3}.
\]
Subtracting \eqref{e3.8} and \eqref{e3.6} we obtain
\[
-6(3-\lambda)\varphi_3 a_3+6(3-\lambda)a_2^2\varphi_3=(p_2-q_2)\gamma,
\]
and using \eqref{e3.10} we deduce
\[
a_3=\frac{(p_2+q_2)\gamma}{8(\lambda^2-2\lambda)\varphi_2^2+6(3-\lambda)\varphi_3}-
\frac{(p_2-q_2)\gamma}{6(3-\lambda)\varphi_3}.
\]
Finally, from \eqref{e3.8} and using \eqref{e3.9} we get
\[
a_3=\frac{1}{3(3-\lambda)}\left(1+\frac{3(3-\lambda)\varphi_3}{2(2-\lambda)^2\varphi^2_2}\right)p_1^2\gamma^2-
\frac{q_2\gamma}{3(3-\lambda)\varphi_3}.
\]

Proceeding on lines similar to the proof of Theorem \ref{Bi-th1} and applying the Lemma \ref{lem2.1}, we get the desired estimate given in \eqref{K-a3}.
\end{proof}

Taking $\lambda=1$ and $\lambda=0$ in Theorem \ref{Bi-th1}, we obtain the following corollaries, respectively:

\begin{corollary}\label{cor1thgms1}
If the function $f$ given by \eqref{Int-e1} belongs to the class $\mathcal{K}^{a,b,c}_{\Sigma}(\gamma,\beta)$, then
\[
|a_2|\leq\min\left\{\sqrt{\frac{m|\gamma|(1-\beta)}{\left|6\varphi_3-4\varphi_2^2\right|}};
\frac{m|\gamma|(1-\beta)}{2\varphi_2}\right\}
\]
and
\begin{eqnarray*}
|a_3|\leq\min\left\{\frac{m|\gamma|(1-\beta)}{6\varphi_3}\left( 1+\frac{m|\gamma|(1-\beta)}{\varphi_2^2}\right);
\frac{m|\gamma|(1-\beta)}{6\varphi_3}+\frac{m|\gamma|(1-\beta)}{\left|6\varphi_3-4\varphi_2^2\right|};\right.\\
\left.\frac{m|\gamma|(1-\beta)}{6\varphi_3}+\frac{m^2|\gamma|^2(1-\beta)^2}{6\varphi_3}
\left(1+\frac{6\varphi_3}{2\varphi_2^2}\right)\right\},
\end{eqnarray*}
where $\varphi_2$ and $\varphi_3$ are given by \eqref{varphi-n}.
\end{corollary}

\begin{corollary}\label{cor2thgms1}
If the function $f$ given by \eqref{Int-e1} belongs to the class $\mathcal{Q}^{a,b,c}_{\Sigma}(\gamma,\beta)$, then
\[
|a_2|\leq\min\left\{\sqrt{\frac{m|\gamma|(1-\beta)}{9\varphi_3}};\frac{m|\gamma|(1-\beta)}{4\varphi_2}\right\}
\]
and
\[
|a_3|\leq\dfrac{m|\gamma|(1-\beta)}{9\varphi_3},
\]
where $\varphi_2$ and $\varphi_3$ are given by \eqref{varphi-n}.
\end{corollary}

\begin{remark}\label{remf1}
For $a=c$ and $b=1$, we have $\varphi_n=1$ for all $n\geq1$, and taking $\gamma=1$ and $m=2$ in Corollary \ref{sss1} and Corollary \ref{Bi-cor1} we obtain more accurate results corresponding to the results obtained in \cite{HMS-GMS-NM,HMS-AKM-PG}.
\end{remark}

\begin{remark}\label{remf2}
(i) If $a=1$, $b=1+\delta$, $c=2+\delta$, with $\operatorname{Re}\delta>-1$, then the operator $I_{a,b,c}$ turns into well-known {\em Bernardi operator}, that is
\[
B_f(z):=\mathcal{I}_{a,b,c}f(z)=\frac{1+\delta}{z^\delta}\int_0^zt^{\delta-1}f(t)\operatorname{d}t.
\]

(ii) Moreover, the operators $\mathcal{I}_{1,1,2}$ and $\mathcal{I}_{1,2,3}$ are the well-known {\em Alexander
and Libera operators}, respectively.

(iii) Further, if we take $b=1$ in \eqref{Hohlov}, then $\mathcal{I}_{a,1,c}$ immediately yields the {\em Carlson-Shaffer operator}, that is $L(a,c):=\mathcal{I}_{a,1,c}$.

Remark that, various other interesting corollaries and consequences of our main results, which are asserted by Theorem \ref{Bi-th1} and Theorem \ref{thgms1} above, can be derived similarly. The details involved may be left as exercises for the interested reader.
\end{remark}
\section*{\textbf{Conflict of Interests}}

The authors declare that there is no conflict of interests regarding the
publication of this paper.


\end{document}